\documentclass[12pt]{amsart}

\usepackage{bbold}
\usepackage{amsmath}
\usepackage{amsthm}
\usepackage{graphicx}
\usepackage{amssymb}
\usepackage{xcolor}
\usepackage{subcaption}

\usepackage [english]{babel}
\usepackage [autostyle, english = american]{csquotes}
\MakeOuterQuote{"}

\newtheorem{thm}{Theorem}[section]
\newtheorem{lemma}[thm]{Lemma}

\newtheorem{corollary}[thm]{Corollary}

\theoremstyle{definition}

\newtheorem{defn}{Definition}[section]

\theoremstyle{remark}
\newtheorem{remark}{Remark}

\title[Wigner type laws for structured random matrices]{Wigner type laws for structured random matrices}
\author{
Tapesh Yadav
}
\begin{document}


\begin{abstract}
    For a sufficiently nice 2 dimensional shape, we define its approximating matrix (or patterned matrix) as a random matrix with iid entries arranged according to a given pattern. For large approximating matrices, we observe that the eigenvalues roughly follow an underlying distribution. This phenomenon is similar to the classical observation on Wigner matrices. We prove that the moments of such matrices converge asymptotically as the size increases and equals to the integral of a combinatorially-defined function which counts certain paths on a finite grid. 
     We also consider the case of several independent patterned matrices. Under a specific set of conditions, these matrices admit asymptotic freeness with respect to full-filled independent square random matrices. 
     In our conclusion, we present several open problems.
\end{abstract}
\maketitle


\section{Introduction}

We begin with a classical result by Wigner. Let $Z_N = (1/\sqrt{N})\{Z_{i,j}\}_{i,j=1}^N$ be a random matrix such that $\{Z_{i,j}\}_{1 \leq i \leq j \leq N}$ are iid standard Gaussian random variables and $Z_{i,j} = Z_{j,i}$ for all $i > j$. Such random matrices are referred to as the Gaussian orthogonal ensemble (GOE). For a matrix $A$ of size $N$, let $tr(A)=\frac{1}{N}Tr(A)$ denote normalized trace of the matrix $A$. The Wigner semicircle law (\cite{mingo2017free}) states that for a sequence $\{Z_N\}_{N=1}^\infty$ of GOE random matrices, and for all non negative integers $k$,
    $$ \lim_{N \to \infty} E(tr(Z_N^k)) = \frac{1}{2\pi} \int_{-2}^2 t^k \sqrt{4-t^2} dt$$
    
In other words, the moments of $Z_N$ goes to the moments of the semicircular distribution. This tells us that for large $N$, the eigenvalue distribution of the random matrix $Z_N$ converges to the semicircular law with perhaps one exceptional large eigenvalue (Figure \ref{fig0}).

\begin{figure}[!]
    \centering
\includegraphics[width=12cm, height=5.5cm]{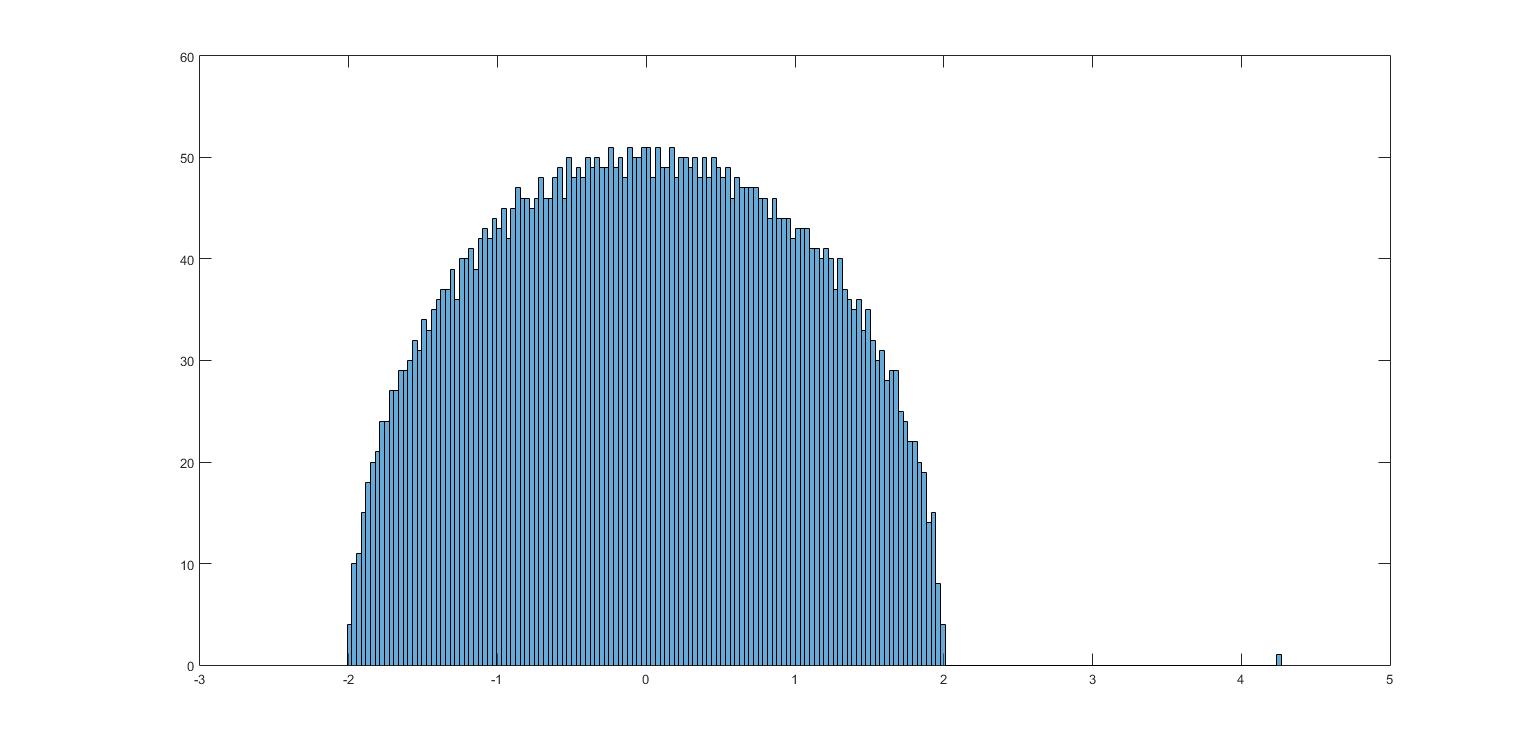}
    \caption{a  histogram of 5000 by 5000 Gaussian self adjoint (real) random matrix (GOE) with non zero mean. Each entry has mean 4/N and standard deviation $\frac{1}{\sqrt{N}}$, N=5000. Note that there is one exceptional large eigenvalue, while the rest follow a semi-circular distribution.}
     \label{fig0}
\end{figure}

In recent years, attempts have been made to understand the behaviour of large random, not necessarily be self adjoint, matrices of different types. In particular, Dykema and Haagerup in \cite{dykema} looked at the distribution limits of the upper triangular random matrices with iid complex standard Gaussian entries in the strictly upper triangular part and iid random variables distributed according to a compactly supported measure $\mu$ on the main diagonal. They showed that their upper triangular matrices converge in distribution to an element in a tracial Von Neumann algebra. The limiting element is called a DT-element or a DT-operator. These operators form an important class of non normal operators for which the Brown measure is explicitly known.  The DT-operators include Voiculescu's circular operator and the circular free Poisson operators. Star moments of of these operators show interesting combinatorial properties (see work by Sniady \cite{sniady}).  Dykema and Haagerup in \cite{dykemainv} proved that every DT operator has a nontrivial, closed, hyperinvariant subspace. Furthermore, every DT-operator generates the Von Neumann algebra $L(\mathbb{F}_2)$ of the free group on two generators.


The moment method is a very useful tool in random matrix theory and free probability (see \cite{mingo2017free},\cite{bose2008another}). The method refers to a variety of ways through which we can understand a measure $\mu$ through the integral of the form (called moments of the measure $\mu$), $$\int x^n d \mu(x).$$
As a classical example, the moments of the GOE random matrices converges to the moments of the semi circular law asymptotically, where the moments of an $N$ by $N$ random matrix $X_N$ is given by $E[tr(X_N^n)]$. In general, given a random matrix $X_N$ of size $N$, we define a unital state $\phi_N$ on $\mathbb{C}\langle x,x^* \rangle$ called the moment map given by  $\phi_N(p(x,x^*)) = E[tr(p(X_N,X_N^*))]$ for every non commutative polynomial $p \in \mathbb{C}\langle x,x^* \rangle$. Given a non commutative probability space $(\mathcal{A},\phi)$ and $X \in \mathcal{A}$, we say that the sequence of random matrices $X_N $ goes to $X$ in distribution asymptotically if $\phi_N(p(x,x^*)) \to \phi(p(X,X^*)) $ as $N \to \infty$.\\

\begin{figure}
  \centering
  \includegraphics[width=7cm, height=5.5cm]{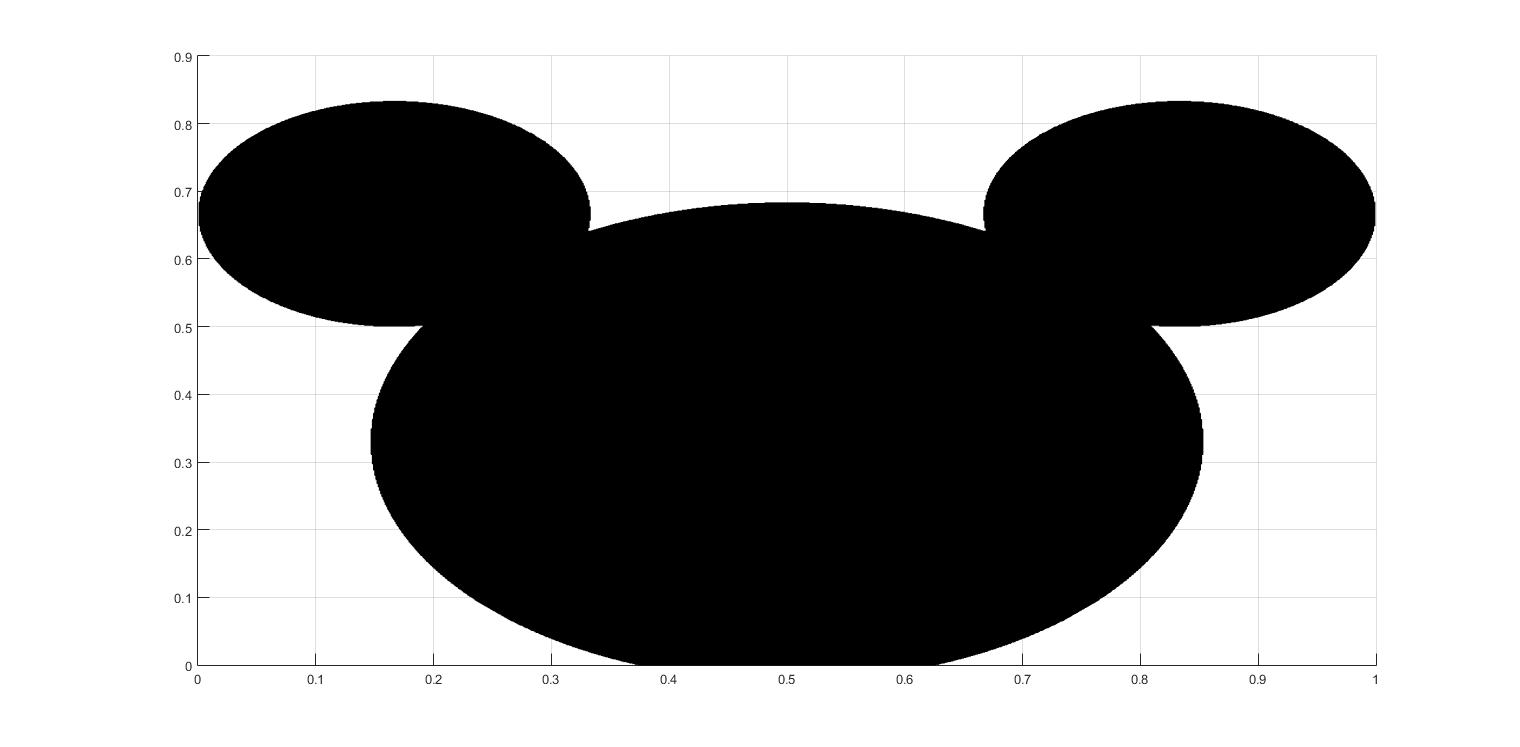}
  \caption{Union of three discs pattern in $[0,1]^2$}
  \label{3disc}
\end{figure}

Our main theorem states that a class of random matrices called the patterned matrices have a limit in distribution. To state the main theorem, we give some definitions. Let $S$ be a \textbf{pattern} i.e. $S \subset [0,1] \times [0,1]$ with boundary of $S$ having Lebesgue measure on $[0,1]^2$ equal to zero (for example, Figure \ref{3disc}).  We denote by $B_N(i_1,i_2,...,i_m) \in [0,1]^m$ the $m$ dimensional cube given by the set $(i_1/N,(i_1+1)/N) \times ... \times (i_m/N,(i_m + 1)/N)$, where $i_k \in \{0,...,N-1\}$ and $ k \in \{1,...,m\}$. Let $X_{N,S}$ be the \textbf{approximating matrix} ($N$ by $N$) of the given pattern $S$ (Figure \ref{fig2}). We define approximating matrix (also called the \textbf{patterned random matrix}) of a pattern by letting $(X_{N,S})_{ij}= X_{ij}/\sqrt{N}$, if $B_N(j-1,N-i) \cap S$ has non empty interior. Otherwise make it zero. Here $\{X_{ij}\}$ is a collection of iid random variables given beforehand, with the assumption that they have finite moments of all orders with \textbf{mean 0} and \textbf{variance 1}. Mean zero is essential in our analysis of random matrices, while variance is restricted to 1 for simplification.\\

Let $C$ be a \textbf{configuration}, which is a finite sequence in $x$ and $x^*$. The $k^{th}$ element of this sequence is denoted $C(k)$ and length of the sequence is called the length of the configuration. For example, $C=xx^*x^*xxx^*$ is of length 6 and $C(3)=x^*$. For a matrix $X$, we denote by $X^C$ the product of matrix with itself and its complex conjugate given by the configuration $C$. For example, if $C=x^*xxx^*$, then $X^C= X^*XXX^*$, where $X^*$ represent adjoint of matrix $X$. Given a configuration $C$ and a pair of random variables, $X_{i,j}$ and $X_{j,i}$, we denote by $X_{i,j}^{C(k)}$ the random variable $X_{j,i}$ if $C(k)=x^*$. If $C(k)=x$, $X_{i,j}^{C(k)}$ denotes the random variable $X_{i,j}$. Notation $X_{i,j}^*$ is also used to denote the random variable $X_{j,i}$. \\

\begin{figure}
    \centering
\includegraphics[width=12cm, height=4cm]{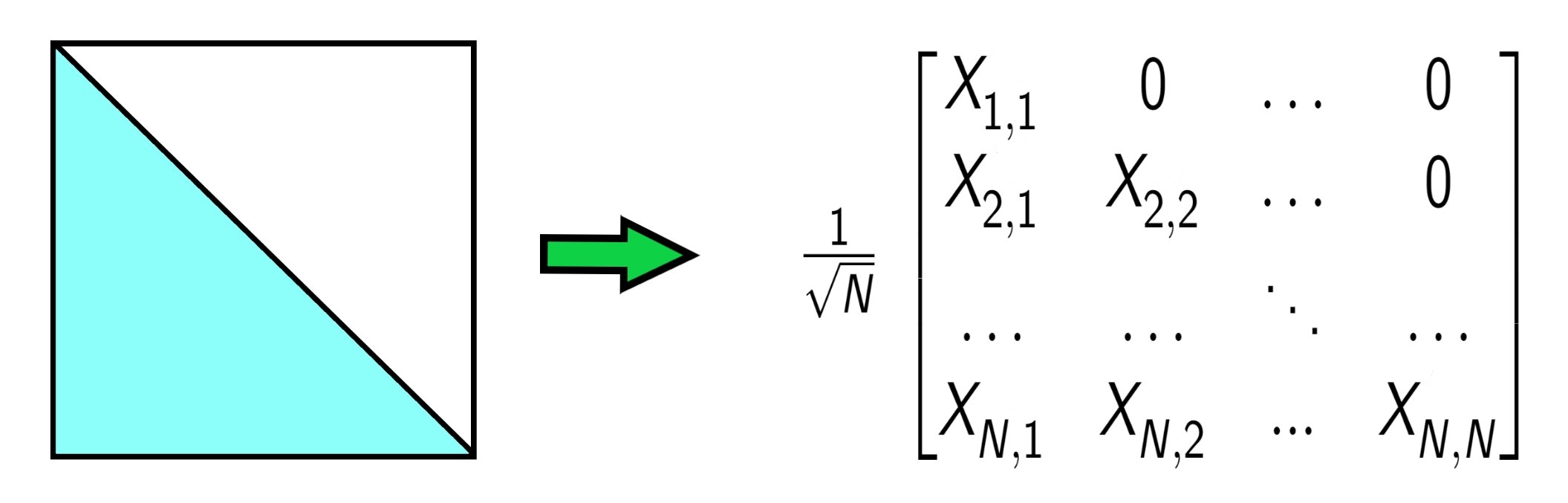}
    \caption{The figure above gives a solid triangular pattern. The approximating matrix for this shape is a lower triangular matrix. }
    \label{fig2}
\end{figure}

We will show that for a pattern $S$, the moments of the approximating matrix under the configuration $C$ converges as the size goes to infinity. The limiting moments are given by the scaled integral of a (combinatorial) function $f_{S,C}^{n+1}:[0,1]^{n+1} \to \mathbb{Z}_+$ called the path counting function, which will be defined later. We state the main theorem below and will prove the same as Theorem \ref{mainthm}.\\

\begin{thm}[Main Theorem] \label{MT}
    Given a configuration C of length $2n$ and a pattern S, asymptotic starred moment for the approximating matrix of pattern S with configuration C exists and is equal to $ lim_{N \to \infty} E[tr(X_{N,S}^C)]$ $= \frac{1}{(n+1)!}\int_{[0,1]^{n+1}}f_{S,C}^{n+1}(x_1,x_2,...,x_{n+1})d\mathbf{x}$.
\end{thm}

The proof relies on basic combinatorics and counting methods. We will see that only the second moment of the underlying random variables contribute to the asymptotic moment calculation for the random matrices, which has been a common theme recently in the area \cite{tao}. Another observation is that the moments do not depend on whether the underlying random variables are real or complex which is a consequence of something what we refer to as polar pairing of paths (see Remark \ref{rem3}). We will start with the definitions in Section \ref{sec2}. We will state and prove the Main Theorem (Theorem \ref{mainthm}) in Section \ref{sec3}. Section \ref{freeness} generalizes the results from section \ref{sec3} to several independent patterned matrices. In the same section, we will show that independent patterned matrices are free from independent fully filled square matrices. We will conclude with Section \ref{sec5}, giving further results and mentioning some open problems. That section includes discussion on the relation of the pattern with its moments, the moments of triangular and fully filled square matrices, affect of base change on the moments and the relation of distribution limits of patterned matrices to finite Von Neumann algebras. We leave open the problem of generalizing our method to patterns with non zero measure boundary. \par
 Von Neumann algebras have been central in the understanding of free probability (see \cite{dykemainv}, \cite{guionnet2010random}), and our results suggest that the distribution limits of patterned matrices live in finite Von Neumann algebra. In the same spirit as  DT operators introduced by Dykema and Haagerup (see \cite{dykema}), the limiting non normal operators for the patterned matrices can help in understanding the spectral properties of operators and their Brown measure.

\section{Preliminaries and Definitions} \label{sec2}

For $n \in \mathbb{N}$, let $[n]$ denote the set $\{1,...,n\}$. Consider an $m+1$ by $N$ rectangular grid, $[m+1] \times [N] \subset \mathbb{Z} \times \mathbb{Z}$ (Figure \ref{fig3}). Each point in the set $[N]$ is referred to as a \textbf{color} and every point in the set $[m+1]$ is referred to as an \textbf{index}. A \textbf{path} on such a grid is a function $i: [m+1] \to [N]$ such that $i(1)=i(m+1)$ (see Figure \ref{fig3}). We sometimes use notation $i_l$ to represent $i(l)$. For a path $i:[m+1] \to [N]$, we define its length to be equal to $m$.

\begin{remark}
Note that our definition of a path is different from the literature as we restrict our paths to have the same end points.
\end{remark}



The \textbf{starred moment} under the configuration $C$ of an $N$ by $N$ random matrix $X_{N,S}$ is given by $E[tr(X_{N,S}^C)] = \frac{1}{N^{\frac{m}{2}+1}}
E[\sum_{i_1,..,i_m=1}^NX_{i_1,i_2}^{C(1)}X_{i_2,i_3}^{C(2)}...X_{i_{m},i_{1}}^{C(m)}]$.  To each \textbf{product term} of the form $X_{i_1,i_2}^{C(1)}X_{i_2,i_3}^{C(2)}...X_{i_{m},i_{1}}^{C(m)}$, we associate a unique path $i:[m+1] \to [N]$  given by $i(k)=i_k$ for all $k \in [m]$ and setting $i(m+1)=i(1)$. We call this path $i$ to be the \textbf{path derived from the product term} $X_{i_1,i_2}^{C(1)}X_{i_2,i_3}^{C(2)}...X_{i_{m},i_{1}}^{C(m)}$ in the
trace expansion of $X_{N,S}^C$. Conversely, given a pattern $S$ and configuration $C$ and a path $i: [m+1] \to [N]$ such that $i(1)=i(m+1)$, we can uniquely reconstruct the product term $X_{i_1,i_2}^{C(1)}X_{i_2,i_3}^{C(2)}...X_{i_{m},i_{1}}^{C(m)}$ by setting $i_k=i(k)$ for all $k \in [m]$. This gives a unique correspondence between paths and the product terms in trace expansion of the matrix $X_{N,S}^C$. 
\\

\begin{figure}[!] 
    \centering
\includegraphics[width=12cm, height=5.5cm]{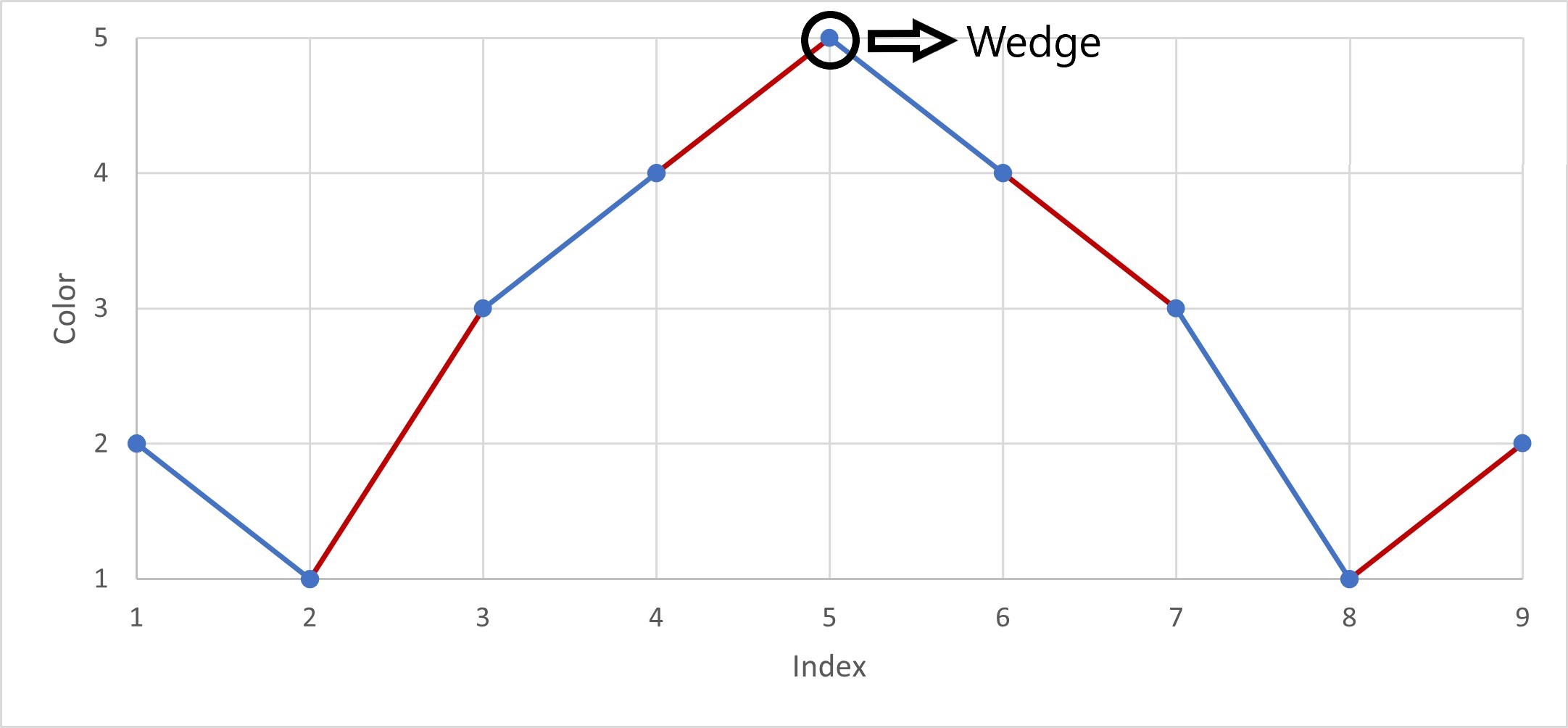}
    \caption{A polar constraint path $i:[8+1] \to [5]$ corresponding to the product term $X_{21}X^*_{13}X_{34}X^*_{45}X_{54}X^*_{43}X_{31}X^*_{12}$ in the trace expansion of $X_{N,S}^C$ where the configuration $C=xx^*xx^* xx^*xx^*$, pattern $S=[0,1]^2$ and N=5. We have marked the wedge.}
    \label{fig3}
\end{figure}

Here we define some specific type of paths. Let $C$ be a configuration of length $m$ and $S$ be a pattern. We define the $k^{th}$ \textbf{edge} of a path $i:[m+1] \to [N]$ to be the set $\{i(k),i(k+1)\}$ for $k\in [m]$ . Note that the $k^{th}$ edge corresponds to $k^{th}$ random variable in the product term derived from the given path in trace expansion of $X_{N,S}^C$. We call an edge $\{i(k),i(k+1)\}$ starred if $C(k)=x^*$ and unstarred otherwise. A path $i:[m+1] \to [N]$, is said to be \textbf{social} if for all $k \in [m]$, there exists $l \neq k$ such that edge $\{i(k),i(k+1) \}$ equals the edge $\{i(l),i(l+1) \}$. We call a social path a \textbf{pairing} if each edge is paired in the above sense an even number of times. We have \textbf{strict pairing} if each edge is paired with exactly one other edge. A \textbf{Polar paired} path is a paired path with the constraint that it has an equal number of starred and unstarred edges. If $m=2n$, we call a path a \textbf{(polar) constraint path} (see Figure \ref{fig3}) under a configuration $C$ and pattern $S$ if the path,
\begin{enumerate}
    \item follows the pattern and configuration of the matrix $X_{N,S}$, i.e. it is a path derived from a product term in the trace expansion of $X_{N,S}^C$ with none of the random variables identically 0 in the product term from which the path is derived.
    \item passes through $n+1$ colors (i.e. cardinality of the range of the path is $n+1$).
    \item is strictly (polar) paired
\end{enumerate}
 Note that social paths are the only paths contributing to the asymptotic moments. Non social paths always have an unpaired random variable and hence have mean zero due to independence, so they do not contribute in moment calculations. We will show in the next Section that the product terms in the trace expansion indexed by constraint paths are the only ones contributing to the asymptotic moment calculation. The number of all other social paths goes to 0 in proportion to constraint paths asymptotically.\\

The key tool for proving theorems here is the existence of a wedge in a constraint path. We define a \textbf{wedge} (see Figure \ref{fig3}) as the upward or downward V shape part of path. By V shape, we mean two consecutive edges $\{i(k),i(k+1)\}$ and $\{i(k+1),i(k+2)\}$ having $i(k)=i(k+2)$. A \textbf{strict wedge} is defined to be a wedge whose edges do not appear anywhere else in the path other than the wedge itself.


\section{Asymptotic moment of Patterned Random matrices} \label{sec3}

 Asymptotic moments of a random matrix under a pattern S and configuration C refers to the limit $\lim_{N \to \infty} E[tr(X_{N,S}^C)]$. When we say a path $i:[m+1] \to [N]$ passes through $k$ colors, we mean that the cardinality of its range is $k$. A point in the domain of a path is sometimes called an index of the path. Let $\lceil x \rceil$ be the least integer greater than or equal to $x$. We begin with certain properties of the paths. 

\begin{lemma} \label{lem1}
    The number of paths $i:[m+1] \to [N]$  normalized by ${1}/{(N^{\frac{m}{2} + 1})}$  passing through at most $\lceil m/2 \rceil$ colors, asymptotically goes to zero.
\end{lemma}


\begin{proof}

Passing through at most $\lceil m/2 \rceil$ colors is the same as saying that cardinality of the range of the function $i$ is at most $\lceil m/2 \rceil$. Total number of such functions $i:[m+1] \to [N]$ are of the order $O(N^{\lceil m/2 \rceil})$. Hence normalizing by $1/N^{\frac{m}{2} + 1}$ gives us the lemma.
\end{proof}

\begin{lemma} \label{lem2}
   The number of social paths $i:[(2n+1)+1] \to [N]$ normalized by ${1}/{(N^{\frac{(2n+1)}{2} + 1})}$ asymptotically goes to 0. This gives that given an odd length configuration C of length $2n+1$, $\lim_{N \to \infty} E[tr(X_{N,S}^C)] = 0$.
\end{lemma}

\begin{proof}

 we will show that social paths in an odd length configuration of length $2n+1$ pass through at most $n+1 $ colors, and apply lemma \ref{lem1} to conclude the result. We will do this by induction.

The base case is obvious. For a social path of length $2n+1$, we know that an edge appears at least thrice. We proceed by analysing two cases. For the first case, assume we do not have any strict wedge in the path. In this case, the inverse image of every point in the range of the path $i:[2n+2] \to [N]$ has cardinality at least two. Since one of the edges appears at least thrice, we have that the function $i$ takes the same value at three distinct points. Thus the range of the path $i$ has cardinality at most $n+1$; i.e. the path passes through at most $n+1$ colors.

Suppose that the path has a strict wedge. In this case, we remove the wedge and join the endpoints together to create a new path of length $2(n-1)+1$. More precisely, suppose $i:[2n+1] \to [N]$ is a path with wedge at point $k$ in the domain. We create a new path $i':[2n-1] \to [N]$ defined by $i'(j)=i(j)$ for all $1 \leq j \leq (k-1)$, and $i'(j)=i(j+2)$ for $k \leq j \leq 2n-1$. This new path satisfies the induction hypothesis, and hence passes through at most $n$ colors. The wedge adds at most one new color to the path $i$, and hence the maximum colors it passes through is $n+1$. This concludes the proof for claim about paths. \par

To see that $\lim_{N \to \infty} E[tr(X_{N,S}^C)] = 0$ for odd length configuration of length $2n+1$, we note that only the social paths contribute to asymptotic moment. We observe that the in the trace expansion, we get a normalization of ${1}/{N^{\frac{(2n+1)}{2} + 1}}$, and hence number of social paths asymptotically goes to 0. Also, mean of every term corresponding to a social path is bounded uniformly. Therefore, we get that $\lim_{N \to \infty} E[tr(X_{N,S}^C)] = 0$.

\end{proof}

\begin{remark} \label{rem1}
The method used in the proof of lemma \ref{lem2} is a fundamental method that we adopt whenever we try to analyse paths. This boils down to inferring property of the paths by dividing them in two cases: the case where we do have a wedge and the other, where we don't. The case without wedge is usually straight forward. In the case of a wedge, we remove the wedge to get a smaller
path of the same nature and conclude using induction.
\end{remark}

\begin{remark}\label{rem2}
Using the same technique as in proof of Lemma \ref{lem2}, we can infer that for a configuration of length $2n$, a social path can pass through at most $n+1$ colors. 
\end{remark}

From now on, all configurations are of even length.

\begin{lemma}[Existence of a wedge for a social path with n+1 colors]\label{lem3}
Let $i:[2n+1] \to [N]$ be a social path and passing through $n+1$ colors. Such a path has a strict wedge. 
\end{lemma}

\begin{proof}
This is true because if there is no strict wedge, each color has to appear at at least two indices as each edge is at least paired. And there are $2n+1$ indices. We observe that one of the colors appear on at least three indices, as there are odd number of indices $2n+1$. Thus, the path passes through at most $n$ colors if there is no strict wedge.
\end{proof}

\begin{remark} \label{rem3.0}
In order for a path of length $2n$ to pass through $n+1$ colors, the path has to be strict pairing. The reason being, by Lemma \ref{lem3}, there has to be a strict wedge in the path. Remove the strict wedge and join its endpoints to get a social path of length $2(n-1)$ with $n$ colors. By an induction argument, this new path is strictly paired and hence the original path is strictly paired.\\

\end{remark}

We will prove an important lemma below which characterizes the paths which contribute to asymptotic moments. 

\begin{lemma} \label{lem4}
    Amongst all paths in the trace expansion of approximating matrix (of pattern $S$) of size $N$ with configuration $C$ (of length 2n), only the polar constraint paths contribute to asymptotic moment as $N \to \infty$.
\end{lemma}

\begin{proof}
From Lemma \ref{lem1}, we know that amongst all the paths in trace expansion of $X_{N,S}^C$, only social paths passing through $n+1$ colors contribute to asymptotic moment (see Remark \ref{rem2}). We will show, using induction on length of the configuration (which equals $2n$), that social paths passing through $n+1$ colors with non zero mean (i.e. $E[X_{i_1,i_2}^{C(1)}X_{i_2,i_3}^{C(2)}...X_{i_{m},i_{1}}^{C(m)}] \neq 0$ for the path $i(k)=i_k$) have to be polar constraint paths.\par
Base case (for $n=1$) is obvious as there are only two possible paths that can pass through 2 colors and they have to be polar constraint for non zero mean. For general $n>1$, we note that there has to be a strict wedge by Lemma \ref{lem3}. We remove this wedge and join the endpoints to get a new path of length $2n$. The new path is still social but with one less color, and hence the new path is a social path of length $2(n-1)$ passing through $n$ colors. If the reduced path is polar constraint, then by the induction hypothesis, this path will have non zero mean. Hence, adding the wedge back makes the original path of length $n$ to have non zero mean only if the wedge represents polar paired random variables. Therefore, the whole path has to be polar paired for non zero mean in this case. \par
In the case when the reduced path is not polar paired and therefore has zero mean, the original path will have zero mean because the removed wedge is a strict wedge.
\end{proof}

\begin{remark} \label{rem3}
Remark \ref{rem3.0} gives that in the asymptotic moment calculation, information about the second moment is enough (as constraint paths have strict pairing). Lemma \ref{lem4} gives that complex and real matrices both have the same moments (due to `polar' pairing).
\end{remark}

\begin{defn}

    Given a pattern $S \subset [0,1] \times [0,1]$, and a configuration C of length $2n$, we define a function in $n+1$ variables, called the \textbf{path counting function}, $f_{S,C}^{n+1} : [0,1]^{n+1} \to \mathbb{N}$ as follows. For distinct $x_1,...,x_{n+1} \in [0,1]$, $f_{S,C}^{n+1}(x_1,...,x_{n+1})$ counts the number of paths $i:[2n+1] \to [n+1]$, with the following constraints, 
    \begin{enumerate}
    \item path $i$ is strictly polar paired and passes through $n+1$ colors.
        \item  for $k \in \{1,...,2n\}$, the ordered pair $(x_{i(k)},x_{i(k+1)}) \in S$ if $C(k)=x$, and $(x_{i(k+1)},x_{i(k)}) \in S$ if $C(k)=x^*$.
    \end{enumerate} 
     We set $f_{S,C}^{n+1}(x_1,...,x_{n+1})=0$ if $x_i=x_j$ for $i \neq j$.\\ 
     
\end{defn}
	
\begin{thm}[Main Theorem] \label{mainthm}
    Given a configuration C of length $2n$ and a pattern S, asymptotic starred moment for the approximating matrix of pattern S with configuration C exists and is equal to $ lim_{N \to \infty} E[tr(X_{N,S}^C)]$ $= \frac{1}{(n+1)!}\int_{[0,1]^{n+1}}f_{S,C}^{n+1}(x_1,x_2,...,x_{n+1})d\mathbf{x}$.
\end{thm}


\begin{figure}
\begin{subfigure}[b]{0.45\textwidth}
  \centering
  \includegraphics[width=0.9\linewidth, height=120 pt]{mickey_mouse_refined.jpg}
  \caption{Union of three discs pattern in $[0,1]^2$ with discs centered at (0.5,0.33), (1/4,1/6) and (5/6,4/6) of radius $\frac{1}{\sqrt{8}}$, $\frac{1}{6}$ and $\frac{1}{6}$}
  \label{fig:sfig1}
\end{subfigure}
\hfill
\begin{subfigure}[b]{.45\textwidth}
  \centering
  \includegraphics[width=1.7\linewidth, height=130 pt]{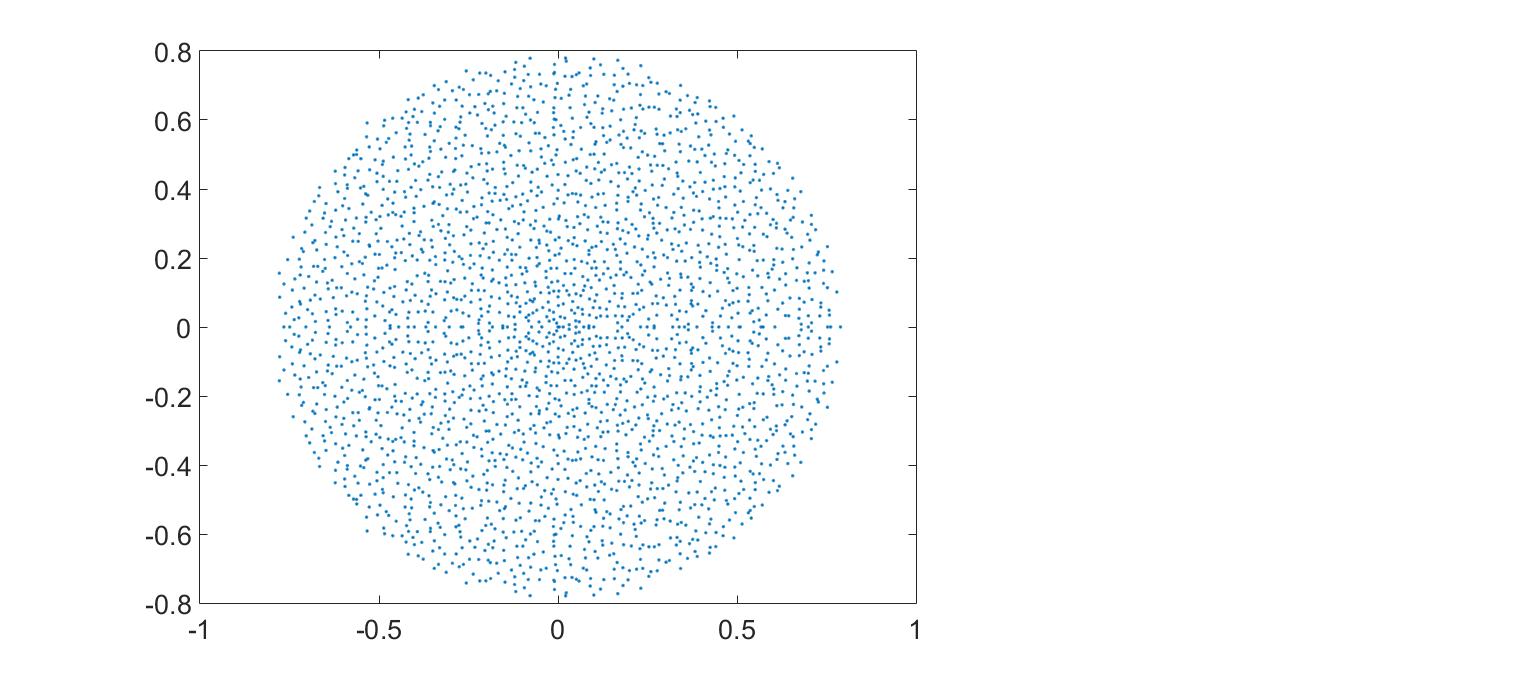}
  \caption{Spectrum for approximating matrix of union of three discs pattern of size 3000}
  \label{fig:sfig2}
\end{subfigure}
\begin{subfigure}{1\textwidth}
  \centering
  \includegraphics[width=1.1\linewidth]{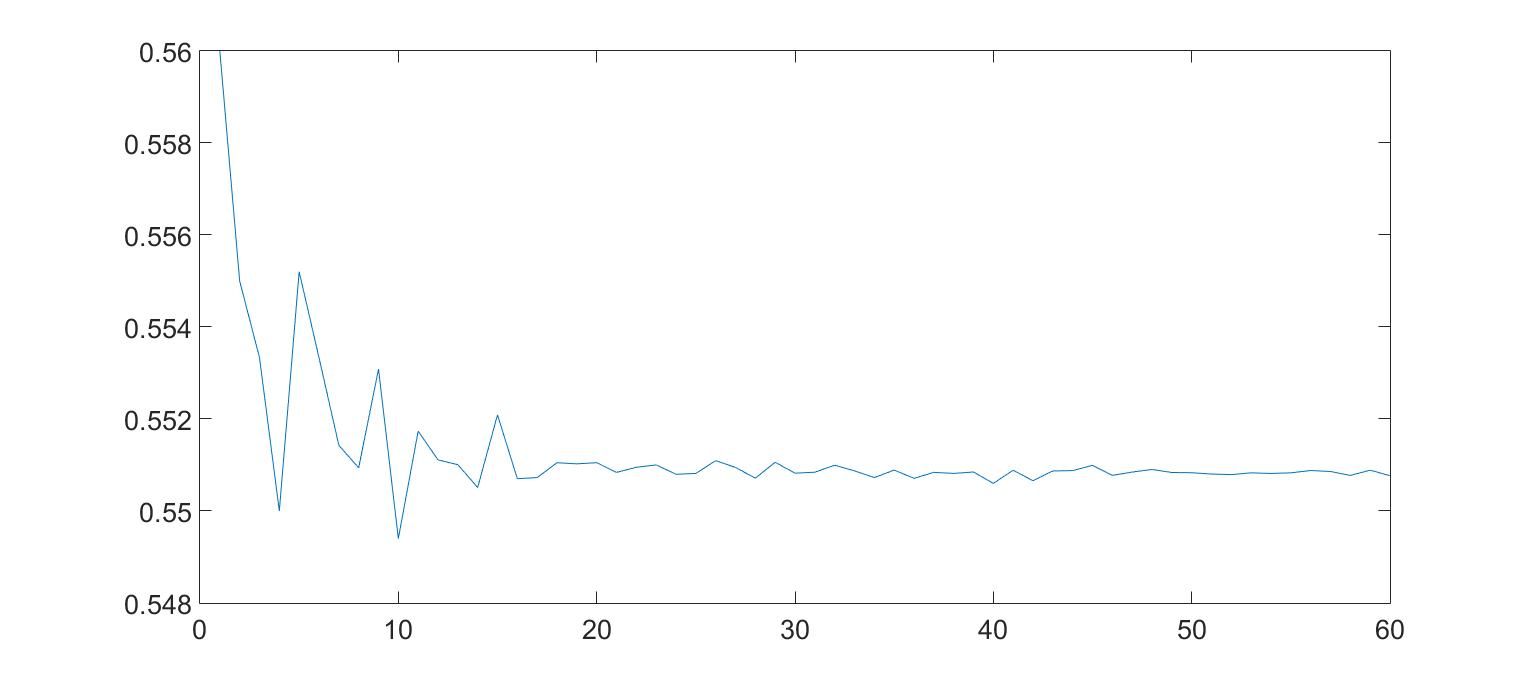}
  \caption{Graph of the moments of (a realization of) approximating matrices under configuration $xx^*$ as size increases. X axis represents $\frac{1}{10}$ the size of approximating matrix and Y axis gives the corresponding moment of the approximating matrix of given size. From the graph, moments can be seen converging to 0.5508}
  \label{fig:sfig3}
\end{subfigure}

\caption{Spectrum and moments of union of three discs pattern}
\label{fig:fig}
\end{figure}

\begin{proof}

 For each $N \in \mathbb{N}$, we construct simple function $s_N$ on $[0,1]^{n+1}$ by partitioning $[0,1]^{n+1}$ into $n+1$ dimensional disjoint cubes of edge length $1/N$, $B_N(i_1,...,i_{n+1})$ for all $i_k \in \{0,...,N-1 \}$. On $B_N(i_1,...,i_{n+1})$, define simple function's value to be equal to the number of polar constraint paths $l:[2n+1] \to [N]$ such that $l(k) \in \{i_1+1,...,i_{n+1}+1\}$ for all $k \in [2n+1]$ in the trace expansion of  $X_{N,S}^C$. By construction, we get that $s_N$ converges almost surely to $f_{S,C}^{n+1}$. The functions, $s_N$ and $f_{S,C}^{n+1}$  are dominated by the total number of distinct paths from $[2n+1] \to [n+1]$. Hence by dominated convergence Theorem, we get that $\int_{[0,1]^{n+1}} s_N \to \int_{[0,1]^{n+1}}f_{S,C}^{n+1}$ under the Lebesgue measure. We know, $\frac{1}{(n+1)!}\int_{[0,1]^{n+1}} s_N = E[tr(X_{N,S}^C)] - \varepsilon_N$, where $\varepsilon_N$ represents normalised number of paths with non zero mean which are not polar constraint, and hence $\varepsilon_N \to 0$ as $N \to \infty$ by Lemma \ref{lem4}. We have to normalize $\int_{[0,1]^{n+1}} s_N$ by $\frac{1}{(n+1)!}$ because the function $s_N(x_1,...,x_{n+1})$ takes the same value on any permutation of the input variables $\{x_1,...,x_{n+1}\}$, which amounts to over counting of constraint paths by a multiplicative factor of $(n+1)!$. Hence, $E[tr(X_{N,S}^C)] \to \frac{1}{(n+1)!}\int_{[0,1]^{n+1}}f_{S,C}^{n+1}$. 
\end{proof}

\begin{remark}\label{rem5}
In the proof of the Corollary \ref{as}, for a term of the form $X_{i_1,i_2}^{C(1)}...X_{i_{2n},i_{1}}^{C(2n)} - E[X_{i_1,i_2}^{C(1)}...X_{i_{2n},i_{1}}^{C(2n)}]$, we will refer to  $i_1,...,i_{2n}$ as its colors.
\end{remark}

\begin{corollary} \label{as}
Given a configuration C of length $2n$ and a pattern S, asymptotic starred moments for the approximating matrix of pattern S with configuration C exists and are equal to $ lim_{N \to \infty} tr(X_{N,S}^C)$ $= \frac{1}{(n+1)!}\int_{[0,1]^{n+1}}f_{S,C}^{n+1}(x_1,x_2,...,x_{n+1})d\mathbf{x}$ almost surely.
\end{corollary}
\begin{proof}
Note that the variance of $tr(X_{N,S}^C)$ is given by \\ $\frac{1}{N^{2n+2}} E[(\sum_{1 \leq i_1,...,i_{2n} \leq N} X_{i_1,i_2}^{C(1)}X_{i_2,i_3}^{C(2)}...X_{i_{2n},i_{1}}^{C(2n)} - E[X_{i_1,i_2}^{C(1)}X_{i_2,i_3}^{C(2)}...X_{i_{2n},i_{1}}^{C(2n)}])^2]$. Expanding square will give sum over product of the terms of the form $(X_{i_1,i_2}^{C(1)}...X_{i_{2n},i_{1}}^{C(2n)} - E[X_{i_1,i_2}^{C(1)}...X_{i_{2n},i_{1}}^{C(2n)}])$ and $(X_{j_1,j_2}^{C(1)}...X_{j_{2n},j_{1}}^{C(2n)} - E[X_{j_1,j_2}^{C(1)}...X_{j_{2n},j_{1}}^{C(2n)}])$ (we refer to product of two terms as a \textit{biproduct}). We will show that the number of biproducts with non zero mean is $O(N^{2n})$. Since the mean of each biproduct is bounded above uniformly, we will get that the variance of $tr(X_{N,S}^C)$ decreases on order of $N^2$, which is summable. Hence first Borel Cantelli lemma will give that the convergence of  $tr(X_{N,S}^C)$ is in a.s. sense.\par
To prove our claim, we note that the mean of biproduct of two terms is zero if there is no random variable in common between the two terms. So, we count only the biproducts having at least one random variable in common. \par

\textbf{\underline{Case 1: biproducts with both the terms social}}\\
Each term in biproduct passes through at most $n+1$ colors. Both the terms together can have at most $2n+2$ colors. Since we have a random variable in common, we have at least two colors appearing in both the product terms. Hence, the number of distinct colors appearing is at most $2n$. Total number of ways we can fill $2n$ colors is bounded above by $N^{2n}$. In each fixed color configuration, we can have at most $(2n)^{4n}$ distinct biproducts. Thus, total number of such biproducts is bounded above by $(2n)^{4n}N^{2n}$ which is $O(N^{2n})$.\par
\textbf{\underline{Case 2: biproducts with at least one non social term}}\\
In this case, we have at least one unpaired random variable in one of the terms in the biproduct. If a term (in the biproduct) of length $2n$ has even number of unpaired random variables, say $2k$, then the number of distinct colors in the term is upper bounded by $(n-k) + l$, $l \leq 2k$, where at most $(n-k)$ distinct colors are added by paired random variables and exactly $2l$ distinct colors (distinct from the colors added by paired ones) are added by unpaired random variables (Adding color by a random variable $X_{i_k,i_{k+1}}$ appearing in the term $X_{i_1,i_2}^{C(1)}X_{i_2,i_3}^{C(2)}...X_{i_{2n},i_{1}}^{C(2n)}$ refers to the color $i_{k+1}$ ). \par
In the biproduct of two terms, let the first term has $2k_1$ unpaired random variables adding $l_1$ colors and the other term has $2k_2$ unpaired random variables adding $l_2$ colors, with $l_2 \geq l_1$. So, the first term has at most $(n-k_1)+l_1$ colors and the other one has at most $(n-k_2)+l_2$ colors. All the unpaired random variables in the second term have to appear in the first term so that the mean of the biproduct do not vanish. So, each of the color appearing in unpaired random variables of second term appear in the first term. There are at least $l_2$ such colors. So, we get that total number of distinct colors in unpaired terms is upper bounded by $([(n-k_1)+l_1]+[(n-k_2)+l_2]) - (l_2)$, which equals, $2n-(k_1+k_2)+l_1-1$ which is less than or equal to $2n$ colors. An argument similar to case 1 will show that sum of such product terms is of the order $O(N^{2n})$. A Similar estimate works for odd number of unpaired random variables as well. Hence, we are done.
\end{proof}


\section{Generalization to Several independent pattern matrices and freeness} \label{freeness} 

 We will use the term \textbf{fully filled square matrix} to refer to the approximating matrix of appropriate size corresponding to the square pattern $S=[0,1]^2$. We will observe that the moments of several independent pattern matrices converge asymptotically. We also get that fully filled independent square matrices are free from independent patterned matrices. We will directly carry forward some results from one variable case (section \ref{sec3}) without proof. We start with definitions for several variable case. \par

 Define a \textbf{generalized configuration} $C^q(x_1,x_1^*,...,x_q,x_q^*)$ (denoted by $C^q$) as a finite sequence in formal variables $\{x_1,x_1^*,...,x_q,x_q^*\}$. An example would be $C^4=x_2x_2x_2x_4^*x_1x_3x_3x_3x_1^*x_1^*$.   Let $C_1,...,C_d$ be configurations (as defined in Section \ref{sec2}) of length $n_1,...,n_d$. Let $n=n_1+...+n_d$. We call the tuple $(C_1,...,C_d)$ a \textbf{multi-configuration} of length $n$.\par
Let $S=(S_1,...,S_q)$ be a tuple of patterns. Let $X=(X_{N,S_1},X_{N,S_2},...,X_{N,S_d})$ be the tuple of approximating matrices (as defined in Section 2) such that all the underlying random variables are independent. We define \textbf{mixed moment} of a tuple of random matrices $X$ under the generalized configuration $C^q$ to be equal to  $E[tr(X_{N,S}^{C^q})]$. Here $X_{N,S}^{C^q}= C^q(X_{N,S_1},X_{N,S_1}^*,...,X_{N,S_q},X_{N,S_q}^*)$ is the matrix that we get after substituting every formal variable $x_j$ by the corresponding random matrix $X_{N,S_j}$ and multiplying them accordingly. We claim that the mixed moments converge as $N \to \infty$ for any choice of generalized configuration (Theorem \ref{genmainthm}). The proof of this claim works exactly the same way as one variable. The sketch of proof is given below.\par


Let $C^q$ be a generalized configuration. Then the trace expansion for mixed moments under $C^q$ is given by $E[tr(X_{N,S}^{C^q}] =  E[\sum_{i_1,...,i_d=1}^N X_{i_1,i_2}^{C^q(1)}......X_{i_{n},i_1}^{C^q(n)}]$, where each term of the form $X_{i_1,i_2}^{C^q(1)}......X_{i_{n},i_1}^{C^q(n)}$ corresponds to a unique path say $i:[n+1] \to [N]$, and vise versa. Just like one variable case, only the terms corresponding to social path contribute to the mixed moment (i.e. $E[tr(.)]=0$ for non social paths). Using Lemma \ref{lem2}, we get that if the length of multi-configuration is odd, the asymptotic moment equals 0. If the length of multi-configuration is even say $2n$, then we can infer using lemma \ref{lem1} that only the social paths that pass through exactly $n+1$ colors contribute to the asymptotic mixed moment. Lemma \ref{lem3} and Remark \ref{rem3.0} gives us that such a path should have a strict wedge and should be strict pairing.  \par

We will omit the details, but after minor modification of Lemma \ref{lem4}, we can show that in the trace expansion for the mixed moments, only polar paired paths contribute to asymptotic mixed moment. By polar pairing here, we mean that a random variable which appear in a random matrix $X_{N,S_k}$ should be paired to its complex conjugate appearing in the adjoint matrix $X_{N,S_k}^*$. With this information, we claim that the mixed moments do converge asymptotically, and similar to the one variable case, the moments converge to an integral of a (combinatorial) function. This function will be called the \textbf{generalized path counting function} (it counts polar constraint paths in generalized sense), and is defined almost the same way as the path counting function introduced in Section 3. 

\begin{defn}

    Given a tuple of patterns $S=(S_1,...,S_q)$ and generalized configuration $C^q$ of length $2n$, we define a function in $n+1$ variables called the \textbf{generalized path counting function}, $f_{S,C^q}^{n+1} : [0,1]^{n+1} \to \mathbb{N}$ as follows. For distinct $x_1,...,x_{n+1} \in [0,1]$, $f_{S,C^q}^{n+1}(x_1,...,x_{n+1})$ counts the number of paths $i:[2n+1] \to [n+1]$, with the following constraints, 
    \begin{enumerate}
    \item path $i$ is strictly polar paired and passes through $n+1$ colors.
        \item  for $k \in [2n]$, if $C^q(k)=x_j$ or $C^q(k)=x_j^*$, the ordered pair $(x_{i(k)},x_{i(k+1)}) \in S_j$ if $C^q(k)$ is not starred, and $(x_{i(k+1)},x_{i(k)}) \in S_j$ otherwise.
    \end{enumerate} 
     We set $f_{S,C}^{n+1}(x_1,...,x_{n+1})=0$ if $x_i=x_j$ for $i \neq j$. \\

\end{defn}

With generalized path counting function defined, we can infer that mixed moments converge asymptotically to the normalized integral of the generalized path counting function, using exactly the same method that we used to prove the main theorem, Theorem \ref{mainthm}. We state this as Theorem \ref{genmainthm} below without proof.

\begin{thm}[Generalization of the Main Theorem] \label{genmainthm}
    Given a generalized configuration $C^q$ (length $2n$) and tuple of patterns $S=(S_1,..,S_q)$. The asymptotic mixed starred moments for a square random matrix with generalized configuration $C^q$ and pattern tuple $S$ exists and is equal to $ lim_{N \to \infty} E[tr(X_{N,S}^{C^q})]$ $= \frac{1}{(n+1)!}\int_{[0,1]^{n+1}}f_{S,C^q}^{n+1}(x_1,x_2,...,x_{n+1})d\mathbf{x}$.
\end{thm}

\vspace{5 pt}

In the development of free probability, asymptotic freeness has been studied extensively (see \cite{voiculescu1995free}, \cite{mingo2016freeness}). We have found that the patterned random matrices also admit asymptotic freeness under certain conditions. To state and prove our result, we shall start with a definition. \par

\begin{defn}
Given two paths $i_1:[n_1+1] \to [N]$ and $i_2:[n_2+1] \to [N]$ such that $i_1(k)=i_2(1)=i_2(n_2+1)$, we define \textbf{join} of the two paths at index $k \in [n_1]$ to be a new path $i_1 \wedge_k i_2: [n_1+n_2+1] \to [N]$ given by 
\begin{equation*}
   i_1 \wedge_k i_2(j) =
    \begin{cases}
      i_1(j), &  \text{ if } 1 \leq j \leq k \\
      i_2(j-k+1),  &  \text{ if } k+1 \leq j \leq k+n_2 \\
      i_1(j-n_2),  & \text{ if }  k+n_2+1 \leq j \leq n_1+n_2+1 \\
    \end{cases}
  \end{equation*}
\end{defn}

\begin{remark}
Let $(C_1,...,C_d)$ be multi-configuration. We use the notation $X_{N,S_1}^{C_1}...\widehat{X_{N,S_k}^{C_k}}...X_{N,S_d}^{C_d}$ to denote the product of random matrices $X_{N,S_1}^{C_1},...,X_{N,S_{k-1}}^{C_{k-1}},X_{N,S_{k+1}}^{C_{k+1}},...,X_{N,S_d}^{C_d}$ which excludes $X_{N,S_k}^{C_k}$.
\end{remark}

\begin{remark}\label{rem6}
Given a fully filled square matrix $X_{N,S}$ and a configuration $C$, we have that the number of polar constraint paths $i:[2n+1] \to [N]$ such that $i(k_1)=s_1$ is equal to the number of polar constraint paths with $i(k_2)=s_2$. This number is equal to the total number of constraint paths divided by $N$.
\end{remark}


\begin{thm}[Asymptotic Freeness] \label{freethm}
Let $S^1=(S_1^1,...,S_{q_1}^1)$ and $S^2=(S_1^2,...,S_{q_2}^2)$ be tuples of patterns in $[0,1]^2 \in \mathbb{R}^2$. Let $(C_1,...,C_{d})$ be a multi-configuration of length $2n_1+...+2n_d=2n$. Let $m_{S_{i_1}^{g_1}}^{C_{1}},...,m_{S_{i_d}^{g_d}}^{C_d}$ be the limits of $E[tr(X_{N,S_{i_1}^{g_1}}^{C_{1}})],...,E[tr(X_{N,S_{i_d}^{g_d}}^{C_d})]$ respectively as $N \to \infty$, where $g_k \in \{1,2\}$ for all $k \in [d]$, $g_k \neq g_{k+1}$, $i_k \in [q_{g_k}]$. If $\{X_{N,S_{i_g}^{g}}\}_{g \in [2], i_g \in [q_g]}$ are independent, then, $E[tr((X_{N,S_{i_1}^{g_1}}^{C_1}- m_{S_{i_1}^{g_1}}^{C_1}) (X_{N,S_{i_2}^{g_2}}^{C_2}- m_{S_{i_2}^{g_2}}^{C_2})...(X_{N,S_{i_d}^{g_d}}^{C_d}- m_{S_{i_d}^{g_d}}^{C_d}))] \to 0$ as $N \to \infty$.
\end{thm}


\begin{proof}
We expand $\lim_{N \to \infty} E[tr((X_{N,S_1^{g_1}}^{C_1}- m_{S_1^{g_1}}^{C_1})...(X_{N,S_d^{g_d}}^{C_d}- m_{S_d^{g_d}}^{C_d}))]$ as 
\begin{equation}\label{sum1}
    \lim_{N \to \infty} E[tr(\sum_{l=0}^d \sum_{A \subset [d], |A|=l} (-1)^{d-l} (\Pi_{i \in A}X_{N,S_i^{g_i}}^{C_i})(\Pi_{j \in A^c}m_{N,S_j^{g_j}}^{C_j}))].
\end{equation}
If we remove the term $X_{N,S_1^{g_1}}^{C_1}X_{N,S_2^{g_2}}^{C_2}...X_{N,S_d^{g_d}}^{C_d}$ from the above sum, the remaining sum follows inclusion-exclusion principle. We claim that the asymptotic moment of $X_{N,S_1}^{C_1}X_{N,S_2}^{C_2}...X_{N,S_d}^{C_d}$ is equal to negative of the asymptotic moment of remaining sum above. \par

To prove the above claim, note that only the polar constraint paths contribute to the asymptotic moment of $X_{N,S_1^{g_1}}^{C_1}X_{N,S_2^{g_2}}^{C_2}...X_{N,S_d^{g_d}}^{C_d}$. We know that each polar constraint path has a wedge. This gives us that for every polar constraint path $i:[2n+1] \to [N]$, we can decompose the path as join of two polar constraint paths $i_1:[2n-2n_k +1] \to [N]$ and $i_2:[2n_k+1] \to [N]$, i.e., $i= i_1 \wedge_{2n_1+...+2n_{k-1}+1} i_2 $ (because otherwise if we keep on removing the wedge recursively, we will reach in finite steps to a polar constraint path which has no wedges, a contradiction). Note that the path $i_1$ comes from trace expansion of the product $X_{N,S_1^{g_1}}^{C_1}...\widehat{X_{N,S_k^{g_k}}^{C_k}}...X_{N,S_d^{g_d}}^{C_d}$, while $i_2$ comes from $X_{N,S_k^{g_k}}^{C_k}$. There are two cases, either $X_{N,S_k^{g_k}}$ is a fully filled square matrix (i.e. $g_k = 2$) or $X_{N,S_k^{g_k}}$ is some other random matrix (i.e. $g_k=1)$.\par 

\textbf{\underline{Case 1: $X_{N,S_k^{g_k}}$ fully filled square matrix}}\\
We will count the total number of polar constraint paths in trace expansion of $X_{N,S_1^{g_1}}^{C_1}X_{N,S_2^{g_2}}^{C_2}...X_{N,S_d^{g_d}}^{C_d}$ of the form $i= i_1 \wedge_{(2n_1+...+2n_{k-1}+1)} i_2$ such that $i_1:[2n-2n_k +1] \to [N]$ and $i_2:[2n_k+1] \to [N]$. For a polar constraint path $i_1$ in trace expansion of $X_{N,S_1^{g_1}}^{C_1}...\widehat{X_{N,S_k^{g_k}}^{C_k}}...X_{N,S_d^{g_d}}^{C_d}$, we can join all polar constraint paths $i_2$ coming from the trace expansion of $X_{N,S_k^{g_k}}^{C_k}$ with the restriction that $i_2(1)=i_2(2n_k+1)=i_1(2n_1+...+2n_{k-1}+1)$. Using Remark \ref{rem6}, we know that total number of such paths $i_2$ is given by the total number of polar constraint paths in trace expansion of $X_{N,S_k^{g_k}}^{C_k}$ divided by N. So, the total number of polar constraint paths of the form $i= i_1 \wedge_{(2n_1+...+2n_{k-1}+1)} i_2$ is given by product of polar constraint paths in $tr(X_{N,S_1^{g_1}}^{C_1}...\widehat{X_{N,S_k^{g_k}}^{C_k}}...X_{N,S_d^{g_d}}^{C_d})$ multiplied with polar constraint paths in $tr(X_{N,S_k^{g_k}}^{C_k})$ divided by N. As $N \to \infty$, the total number of such paths after normalizing by $\frac{1}{N^{n+1}}$ equals $\lim_{N \to \infty} E[tr(X_{N,S_1^{g_1}}^{C_1}...\widehat{X_{N,S_k^{g_k}}^{C_k}}...X_{N,S_d^{g_d}}^{C_d})](m_{S_{i_k}^{g_k}}^{C_k})$. This is one of the terms in (\ref{sum1}). \par

\textbf{\underline{Case 2: $X_{N,S_k^{g_k}}$ not fully filled square matrix}}\\
In this case, when the matrix at $k^{th}$ position is not a fully filled square, we have that its neighbouring matrices have to be square. Using this fact, we can again decompose the polar constraint path as join of two paths $i= i_1 \wedge_{(2n_1+...+2n_{k-1}+1)} i_2$ in the same way as in case 1, and we get the same asymptotic estimate by using Remark \ref{rem6}. \par

\vspace{4pt}
 From cases 1 and 2 combined, summing over all $k \in [d]$ for decomposed polar constraint paths of the form $i= i_1 \wedge_{(2n_1+...+2n_{k-1}+1)} i_2$ and taking the limit gives $\lim_{N \to \infty} \sum_k E[tr(X_{N,S_1^{g_1}}^{C_1}...\widehat{X_{N,S_k^{g_k}}^{C_k}}...X_{N,S_d^{g_d}}^{C_d})(m_{S_{i_k}^{g_k}}^{C_k})]$. Every polar constraint path has such a decomposition, but the decomposition is not unique as the same path can be written as join of two paths at different indices $k$. This leads to over counting. Hence, by using inclusion exclusion principle, and using similar arguments as above, we get that the exact number of polar constraint paths (normalized by $N^{n+1}$) ) is given, asymptotically, by the negative of the sum of all the terms in (\ref{sum1}) other than the term $\lim_{N \to \infty} E[tr(X_{N,S_1^{g_1}}^{C_1}X_{N,S_2^{g_2}}^{C_2}...X_{N,S_d^{g_d}}^{C_d})]$. $\lim_{N \to \infty} E[tr(X_{N,S_1^{g_1}}^{C_1}X_{N,S_2^{g_2}}^{C_2}...X_{N,S_d^{g_d}}^{C_d})]$ counts all polar constraint paths asymptotically (normalized by $N^{n+1}$). Therefore sum (\ref{sum1}) equals 0 asymptotically.

\end{proof}

\section{Further results and Discussion}     \label{sec5}

\subsection{Relation between configuration and moments} \label{sec4.1}

Let $C$ be a configuration of length $m$. Lemma \ref{lem2} shows that if $m$ is odd, then the asymptotic moments of the approximating matrix under given configuration are zero. For non zero asymptotic moments, the length of the configuration has to be even, say $2n$. Remark \ref{rem3} tells us that the moment is non zero only if the star and non star entries in the configuration appear an equal number of times (see Figure \ref{fig:sfig4}). We write it as theorem

\begin{figure}
  
  \includegraphics[width=12cm, height=5.5cm]{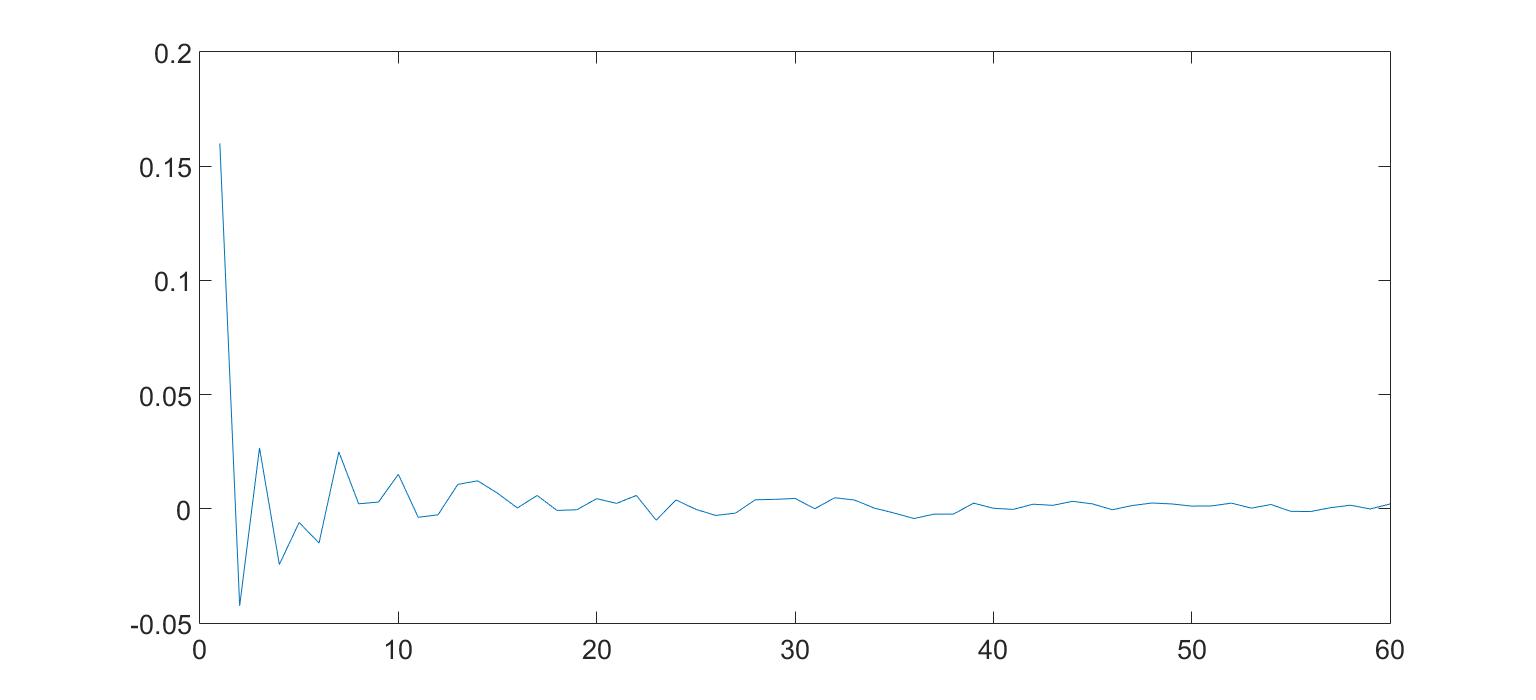}
  \caption{Graph of the moments of (a realization of) approximating matrices of union of three discs pattern (see Figure \ref{fig:fig}) under configuration $xx^*x$. X axis represents $\frac{1}{10}$ the size of approximating matrix and Y axis gives the corresponding moment of the approximating matrix of given size. From the graph, moments can be seen to be converging to 0. This configuration has unequal number of starred and unstarred entries.}
  \label{fig:sfig4}
\end{figure}

\begin{thm}
    Given a configuration and a pattern S with non empty interior, if the number of star entries in the configuration is different from the unstarred entries, then the asymptotic moment $\lim_{N \to \infty} E[tr(X_{N,S}^C)]$ is zero.  
\end{thm}

If the pattern $S$ has non empty interior, then the asymptotic starred moment corresponding to the configuration $xx^*$ is always non zero. For the proof, pick an interior point say $z \in S$. There will be a non zero Lebesgue measure rectangle  $[a_1,a_2] \times [b_1,b_2] \subset S$ containing $z$. On this rectangle, the path counting function $f_{S,xx^*}^{1+1}(x_1,x_2)$ is bounded below by 1 for all $x_1\in [a_1,a_2]$ and $x_2 \in [b_1,b_2]$. Since, $f_{S,xx^*}^{1+1}(x_1,x_2)$ is non negative on its domain, and is bounded below by 1 on a non zero measure set, we have that  $\frac{1}{(n+1)!}\int_{[0,1]^{n+1}}f_{S,xx^*}^{1+1}(x_1,x_2)d\mathbf{x}$ is non zero, which is equal to the asymptotic moment by the main theorem.



\subsection{Combinatorics: Triangular and fully filled square matrices}

The function $f_{S,C}^{n+1}$ is constant on its domain for triangular and fully filled square matrices. This means that the asymptotic moments converge to 
$\frac{1}{(n+1)!}f_{S,C}^{n+1}$. The function $f_{S,C}^{n+1}$ counts the number of constraint paths on $2n+1$ by $n+1$ grid. So, if we can calculate the moments of the random matrix using alternate methods, like using methods from free probability (see \cite{dykema}), we can get hold on the combinatorial quantity given by the function $f_{S,C}^{n+1}$. \\

For triangular matrices, this combinatorial quantity has been explicitly computed using tools from free probability (\cite{dykema},\cite{sniady}) for certain configurations. For example, in \cite{dykema}, they count moments for configurations of the form $xx^*xx^*xx^*..xx^*$, where $xx^*$ repeats $n$ times. This number turns out to be $n^n/(n+1)!$. In \cite{sniady}, we see that for configuration of the form  $(x...xx^*...x^*)...(x...xx^*...x^*)$, where the block $x...xx^*...x^*$ appears $n$ times, and the block consists of $x$ appearing $k$ times consecutively followed by same number of appearances of $x^*$, the value of the path counting function turns out to be $n^{nk}/(nk+1)!$. \\

For fully filled square matrices, it is well known that under the configuration $(xx^*)^n$, the moments equals $n^{th}$ moment of Marchenko–Pastur law. We explicitly know the limiting empirical spectral distribution (ESD) in case of iid entries for such matrix (see\cite{tao}). 


\subsection{Pattern and base change}

We call a pattern $S_1$ to be similar to pattern $S_2$ if the approximating matrices for $S_1$ are similar to the approximating matrices for pattern $S_2$ infinitely often as the size of matrices increases. Since trace is basis independent, we get the same moments for similar matrices. Hence, the moments corresponding to similar patterns are the same. \par

A natural question is to characterize similar patterns (for example, existence of kernel functions relating two patterns) without going to the level of approximating matrices. Another question would be to understand if similar shapes produce a unique collection of starred moments. That is, if two patterns are not similar, do we always get non identical starred moments.\\

\subsection{Approximating matrices for patterns with non zero measure boundary}

Our construction for approximating matrix using interior points works well for patterns with zero measure boundary. If the boundary has non zero measure, our main theorem (as stated in this paper) will not hold true. This is because the approximating matrix will not be a good representation of the given pattern if we follow our method. It remains an open problem to find a way to create more sensible approximating matrix out of the given pattern, which may have non zero measure boundary. One possibility is to use the theory of fractals in this context.


\subsection{Moment reconstruction}

The number of paired paths from $[2n+1]$ to $[n+1]$ is bounded above by $K(n+1)^{n+1}$, where $K$ is a universal constant. This tells us that the path counting function is bounded by $k(n+1)^{n+1}$ for a configuration of length $2n$. Using the Stirling approximation, and the fact that the moments are given by $\frac{1}{(n+1)!}\int_{[0,1]^{n+1}}f_{S,C}^{n+1}$, we conclude that the moments of random matrix under configuration of length $2n$ increases exponentially i.e. moments are bounded by $Ce^n$, for some constant $C$. This gives us that the moments follow the Carleman's condition, and hence we have a unique measure with those moments. Also, the growth rate gives us that the measure has a bounded support. This tells us that we can find an operator in a type II$_1$ factor with the same starred moments as of the limiting moments of the random matrix. These limiting operators are referred to as distribution limits of the given sequence of random matrix. In Dykema's work (see \cite{dykema}), they showed the existence of such operators concretely for triangular matrices and called them the DT operators.



	\bibliographystyle{unsrt}
	\bibliography{main}	

\end{document}